\documentclass[11pt]{article}

\usepackage{amsmath, amsthm, amscd, amssymb, amsfonts, graphicx, fancyhdr, setspace,  enumerate,cancel, enumerate, mathrsfs,latexsym, setspace, tikz,url, ulem}
\usepackage[active]{srcltx}
\usepackage[pagebackref=true]{hyperref}
\newtheorem{theorem}{Theorem}[section]
\newtheorem*{theorem*}{Theorem}
\newtheorem{proposition}{Proposition}[section]

\newtheorem{remark}{Remark}[section]
\theoremstyle{definition}
\newtheorem{definition}{Definition}[section]
\newtheorem{example}{Example}[section]
\newtheorem{claim}{Claim}[section]

\numberwithin{equation}{section}
\newcommand{\C}{\mathcal{C}}
\newcommand{\bb}{\mathcal{B}}
\newcommand{\D}{\mathcal{D}}
\newcommand{\8}{\infty}
\newcommand{\h}{\mathcal{H}(\D)}

\newcommand{\cphi}{C_\varphi}

\date{}
\begin{document}

\title{Eigenfunctions of Composition Operators on Bloch-type Spaces}

\author{ Bhupendra Paudyal }

\maketitle

\begin{abstract}

\noindent Suppose $\varphi$ is a holomorphic self map of the unit disk and $\cphi$ is a composition operator with symbol $\varphi$ that fixes the origin and $0<|\varphi'(0)|<1$. This work explores sufficient conditions that ensure all holomorphic solutions of Schr\"oder equation for the composition operator $\cphi$ belong to a Bloch-type space $\bb_\alpha$ for some $\alpha>0$. The results from composition operators have been extended to weighted composition operators in the second part of this work.
 
\end{abstract}
\maketitle
\section{Introduction} 

Let $\D$ be the unit disk of the complex plane $\C$. Suppose that $\h$ denotes space of holomorphic functions defined on the unit disk. Recall that a holomorphic function $f$ on $\D$  said to be in Bloch-type space  $\bb_\alpha$ for some $\alpha>0$ if 
 \[\sup_{z\in\D} (1-|z|^2)^\alpha |f'(z)|<\8.\] 
 
 Under the norm \begin{equation}\label{1.04}
\|f\|_{\bb_\alpha}=|f(0)|+\sup_{z\in\D} (1-|z|^2)^\alpha |f'(z)|,\end{equation} $\bb_\alpha$ becomes a Banach space. 
From the definition of Bloch-type spaces, it immediately  follows that $\bb_\alpha\subset \bb_\beta$ for $\alpha \leq \beta$ and  $\bb_\alpha \subset H^\8$ for $\alpha <1$.
 
Functions in the Bloch space have been studied extensively by many authors, see \cite{Ar} and \cite{Zh9}. It has been shown in \cite{Zh9} that the Bloch-type norm for $\alpha>1$ is equivalent to the $\alpha-1$ Lipschitz-type norm: 
 \begin{equation}
\label{PRD1}\|f\|_{\bb_\alpha}\approx \sup_{z\in\D} (1-|z|^2)^{\alpha-1} |f(z)|,\hspace{1cm} f\in\bb_\alpha, ~\alpha>1.
\end{equation}

Composing functions $f$ in $\h$, with any holomorphic self-map $\varphi$ of $\D$, induces a linear transformation, denoted by $\cphi$ and called a {\it composition operator} on $\h$:
 \[\cphi f=f\circ \varphi .\]  
 For any $u\in\h$ we define {\it weighted composition operator} $u\cphi$ on $\h$ as  
 \[ u\cphi(f)=(u)(f\circ \varphi).\] 
In this work, we study holomorphic solutions $f$ of the Schr\"oder's equation  
\begin{equation}
\label{Sceq}
(\cphi)f(z) =\lambda f(z),
\end{equation}
and of the weighted Schr\"oder's equation 
 \begin{equation} \label{eqshw}
u\cphi f=\lambda f,
 \end{equation} 
 where $\lambda$ is a complex constant. Assuming $\varphi$ fixes the origin and $0<|\varphi'(0)|<1$, K\"onigs in \cite{ko84} showed that the set of all holomorphic solutions of Eq. \eqref{Sceq} (eigenfunctions of $\cphi$ acting on $\h$) is exactly $\displaystyle\{\sigma ^n\}_{n=0}^\8$, where $\sigma$, principal eigenfunction of $\cphi$, is called {\it K\"onigs function of $\varphi$}. Following the K\"onigs's work, Hosokawa and Nguyen in \cite{Ho10} showed that the set of all eigenfunctions of $u\cphi$ acting on $\h$ is exactly $\{v\sigma^n\}_{n=0}^\infty$ where $v$ is principal eigenfunction of $u\cphi$ and $\sigma$ is the K\"onigs function. \\

According to a general result of Hammond in \cite{ha03} if $u\cphi$ is compact on any Banach space of holomorphic functions on $\D$ containing the polynomials, all eigenfunctions $v\sigma^n$ belong to the Banach space. Hosokawa and Nguyen in \cite{Ho10} under somewhat strong
restrictions on the growths of $u$ and $\varphi$ near the boundary of the unit disk showed that all the eigenfunctions $v\sigma^n$ are eigenfunctions of $u\cphi$ acting on the Bloch space $\bb$. Our goal in this work is to study conditions under which all eigenfunctions $v\sigma^n$ belong to a Bloch-type space $\bb_\alpha$. \\
 
The basic organization of this paper is as follows. We present results concerning to composition operators in Section \ref{compopt}. Theorem \ref{co15} provides the sufficient conditions that ensure all the eigenfunctions $\sigma^n$ belong to Bloch type spaces $\bb_\alpha$ for $\alpha<1$. Similar results for $\alpha=1$ and $\alpha>1$ are presented by Theorem \ref{thm2222} and \ref{thm3.3} respectively. Towards the end of this work we prove results concerning to weighted composition operators.

\section{Preliminaries} \label{prelim}

     We recall the following criteria from \cite[Theorem 2.1]{ZH03} for boundedness of $u\cphi$ on Bloch-type spaces $\bb_\alpha$.
   \begin{theorem} \label{Zh21}
  Let $u$ be analytic on $\D$, $\varphi$ an analytic self-map of $\D$ and $\alpha$ be a positive real number.
\begin{enumerate}
\item If $0<\alpha<1$, then $u\cphi$ is bounded on $\bb_\alpha$  if and only if $u\in\bb_\alpha$ and \[\sup_{z\in \D} |u(z)|\frac{(1-|z|^2)^\alpha}{(1-|\varphi(z)|^2)^\alpha}|\varphi'(z)|<\8.\]

\item The operator $u\cphi$ is bounded on  $\bb$ if and only if
\begin{enumerate}
\item $\sup_{z\in\D}|u'(z)| (1-|z|^2) \log \frac{1}{1-|\varphi(z)|^2}<\8$
\item $\sup_{z\in \D} |u(z)|\frac{(1-|z|^2)}{(1-|\varphi(z)|^2)}|\varphi'(z)|<\8.$
\end{enumerate} 
 
 \item If $\alpha>1$, $u\cphi$ is bounded on $\bb_\alpha$ if and only if the following are satisfied.
 \begin{enumerate}
 \item $\sup_{z\in \D} |u'(z)|\frac{(1-|z|^2)^\alpha}{(1-|\varphi(z)|^2)^{\alpha-1}}<\8$
 \item $\sup_{z\in \D} |u(z)|\frac{(1-|z|^2)^\alpha}{(1-|\varphi(z)|^2)^{\alpha}}|\varphi'(z)|<\8.$
 \end{enumerate}
 
\end{enumerate}

   \end{theorem}

The following theorem from \cite[Theorem 3.1]{ZH03} provides the compactness criterion for $u\cphi$ acting on $\bb_\alpha$.

  \begin{theorem} \label{COMPACTNESS1}
  Let $u$ be holomorphic function on $\D$ and let $\varphi$ be holomorphic self map of $\D$. Let $\alpha$ be a positive real number, and $u\cphi$ is bounded on $\bb_\alpha$.
  \begin{enumerate}
  \item If $0<\alpha<1$ then $u\cphi$ is compact on $\bb_\alpha$ if and only if \[\lim_{ |\varphi(z)|\rightarrow 1^-} |u(z)|\frac{(1-|z|^2)^\alpha}{(1-|\varphi(z)|^2)^\alpha}|\varphi'(z)|=0.\]
  \item The operator $u\cphi$ is compact on $\bb$ if and only if the following are satisfied. 
  \begin{enumerate}
  \item $\lim_{ |\varphi(z)|\rightarrow 1^-}|u'(z)|(1-|z|^2) \log\frac{1}{(1-|\varphi(z)|^2)}=0$
  \item $\lim_{ |\varphi(z)|\rightarrow 1^-} |u(z)|\frac{(1-|z|^2) }{(1-|\varphi(z)|^2)} |\varphi'(z)|=0.$
  \end{enumerate}
  \item If $\alpha>1$, then $u\cphi$ is compact on $\bb_\alpha$ if and only if the following are satisfied.
  \begin{enumerate}
  \item $\lim_{ |\varphi(z)|\rightarrow 1^-} |u'(z)|\frac{(1-|z|^2)^\alpha }{(1-|\varphi(z)|^2)^{\alpha-1}} =0$
  \item 
  $\lim_{ |\varphi(z)|\rightarrow 1^-} |u(z)|\frac{(1-|z|^2)^\alpha }{(1-|\varphi(z)|^2)^\alpha} |\varphi'(z)|=0.$
  \end{enumerate}
\end{enumerate}   
\end{theorem}

\begin{remark} If we assume $u\equiv 1$ in Theorem  \ref{Zh21} and Theorem \ref{COMPACTNESS1}, they provide the criterion for boundness and compactness of composition operators $\cphi$ acting on Bloch-type spaces $\bb_\alpha$.
  
\end{remark}  

The following two theorems are fundamental for our work. Theorem \ref{KOTH} is the famous K\"onigs's Theorem about the solutions to Schr\"oder's equations (see \cite{ko84} and \cite[Chapter 6]{Sh93}).

\begin{theorem}[K\"onigs's Theorem  (1884)]
\label{KOTH}   Assume $\varphi$ is a holomorphic self map of $\D$ such that $\varphi(0) = 0$ and $0 < |\varphi '(0)| < 1$. Then the following assertions hold.
\renewcommand\theenumi{\roman{enumi}}
\begin{enumerate}
\item The sequence of functions
\[\sigma_k(z) := \frac{\varphi_k(z)}{\varphi '(0)^k},\] where $\varphi_k$ is the $k^{th}$ iterates of $\varphi$,
converges uniformly on a compact subset of $\D$ to a non-constant 
function $\sigma$ that satisfies \eqref{Sceq} with $\lambda =\varphi'(0)$.
\smallskip
\item  $f$  and $\lambda$ satisfy \eqref{Sceq}
if and only if there is a positive integer $ n $ such that $\lambda = {\varphi '(0)}^n$ and $f$ is a
constant multiple of $ \sigma^n $. 
\end{enumerate}
\end{theorem}

 The following theorem characterizes all eigenfunctions of a weighted composition operator under some restriction on the symbol (see \cite{Ho10}).
 
  \begin{theorem} 
\label{th22}
Assume $\varphi$ is a holomorphic self map of $\D$ and $u$ is a holomorphic map of $\D$ such that $u(0)\neq 0, \varphi(0)=0, 0<|\varphi'(0)|<1$. Then, the following statements hold.
\renewcommand\theenumi{\roman{enumi}}
\begin{enumerate}
\item The sequence of functions
\[v_k(z)= \frac{u(z)u(\varphi(z))...u(\varphi_{k-1}(z))}{u(0)^k}\]
where $\varphi_k$ is the $k^{th}$ iterates of $\varphi$,
converges to a non-constant holomorphic function $v$ of $\D$ that satisfies \eqref{eqshw} with $\lambda =u(0)$.
\item  $ f$ and $\lambda$ satisfy 
\eqref{eqshw} if and only if 
$ f = v \sigma^n$, $\lambda = u(0)\varphi'(0)^n$, where $n$ is a non-negative integer and $\sigma$ is the solution
of the Schr\"oder equation \eqref{Sceq} $\sigma\circ \varphi=\varphi'(0)\sigma$.
\end{enumerate}
\end{theorem}

\section{Composition operators}\label{compopt}

In this section, we investigate sufficient conditions that ensure the eigenfunctions  $\sigma^n$ of a composition operator belong to $\bb_\alpha$ for some positive number $\alpha$ and for all positive integers $n$.

\begin{definition}
Let us define the {\it Hyperbolic $\alpha$-derivative} of $\varphi$ at $z\in \D$ by
\[\varphi^{(h_\alpha)}(z)= \displaystyle{\frac{(1-|z|^2)^\alpha~\varphi'(z)}{(1-|\varphi(z)|^2)^\alpha}}. \]
\end{definition}
When $\alpha=1 $ then it is simply called the Hyperbolic derivative of $\varphi$ at $z$ and denoted by $\varphi^{(h)}(z)$.
\begin{definition}
 Suppose $\varphi$ is a holomorphic self map of $\D$, $\varphi(0)=0$, $0<|\varphi'(0)|<1$ and $\varphi_m$ is the $m^{th}$ iteration of $\varphi$ for some fixed non-negative integer $m$. Then we say $\varphi$ satisfies condition \eqref{A} if there exists a non-negative integer $m$ such that
\begin{equation}{\nonumber}
\tag{A}
 |\varphi^{(h_\alpha)}(\varphi_{m}(z))|=\displaystyle{\dfrac{(1-|\varphi_m(z)|^2)^\alpha~|\varphi'(\varphi_m(z))|}{(1-|\varphi_{m+1}(z)|^2)^\alpha}}\leq |\varphi'(0)|,
\label{A}
\end{equation}
~ for all~ $z\in \D$~\text{and~for~ some~ fixed~} $\alpha>0$.

\end{definition} 

\begin{remark}
If $\varphi$ satisfies the condition \eqref{A} for some $m$ then it satisfies the condition for all non-negative integers greater than $m$.
\end{remark}

The following example provides a family of maps that satisfies condition \eqref{A}. The example is extracted from \cite{Ho06}.
\begin{example} \label{Ex1}
Consider a map $\gamma$ that maps the  unit disk univalently to the right half plane. This map is given by\[\displaystyle{\gamma(z)=\frac{1+z}{1-z}}.\] For any $t\in(0,1)$, define \[\displaystyle{\varphi_t(z)=\frac{\gamma(z)^t-1}{\gamma(z)^t+1}}.\]  It is well known that $\varphi_t$ maps the unit disk into the unit disk for each $t\in(0,1)$, see \cite{Sh93} . These maps are known as {\it lens map}.\\

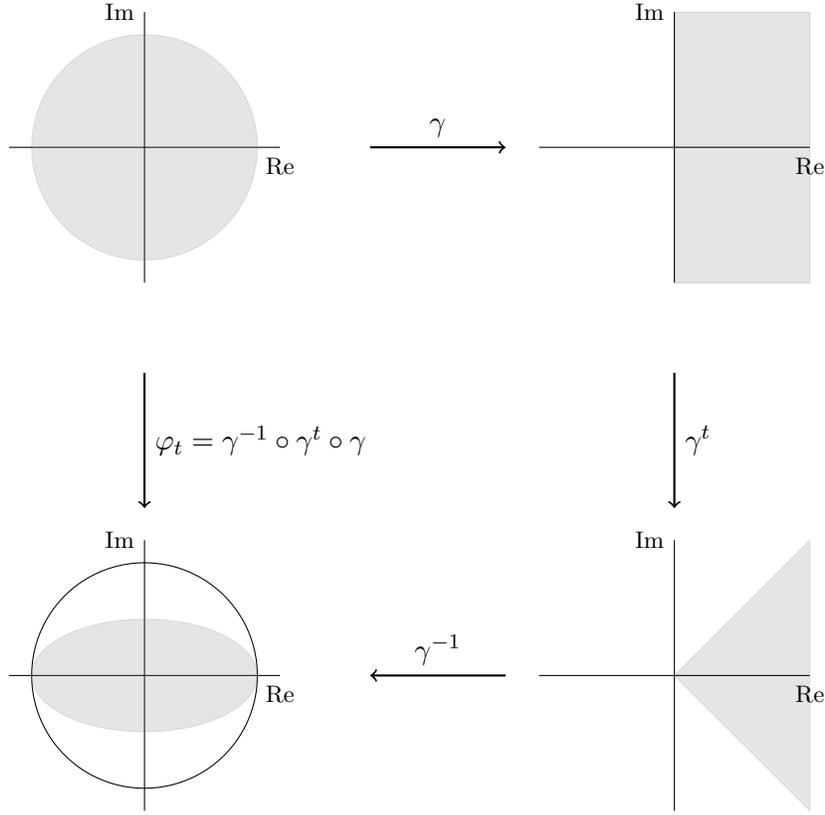
\begin{tabular}{cc}
\begin{tikzpicture}[scale=.6]
\draw[-] (0,-3) -- (0,3)node[left]{\footnotesize Im};
\draw[-] (-3,0) -- (3,0)node[below]{\footnotesize Re};
\draw [ draw=gray, fill=gray, opacity=0.2,domain=0:360,samples=500] plot ({2.5*cos(\x)}, {2.5*sin(\x)});
\draw [thick,->] (5, 0)--(8,0) node[above,midway]{$\gamma$ };
\draw [thick,->] (0, -5)--(0,-8) node[right,midway]{$\varphi_t=\gamma^{-1}\circ \gamma^t\circ \gamma$ };
\end{tikzpicture} 
&
\begin{tikzpicture}[scale=.6]
\fill [ draw=gray, fill=gray, opacity=0.2] (0,-3) rectangle (3,3);
\draw[-] (0,-3) -- (0,3)node[left]{\footnotesize Im};
\draw[-] (-3,0) -- (3,0)node[below]{\footnotesize Re};
\draw [thick,->] (0, -5)--(0,-8) node[right,midway]{$\gamma^{t}$ };
\end{tikzpicture}
\end{tabular}

\begin{tabular}{cc}
\begin{tikzpicture}[scale=.6]
\draw[-] (0,-3) -- (0,3)node[left]{\footnotesize Im};
\draw[-] (-3,0) -- (3,0)node[below]{\footnotesize Re};
\draw [domain=0:360,samples=500] plot ({2.5*cos(\x)}, {2.5*sin(\x)});
\draw [ draw=gray, fill=gray, opacity=0.2, domain=0:360,samples=500] plot ({2.5*cos(\x)}, {1.25*sin(\x)});
\draw [thick,<-] (5, 0)--(8,0) node[above,midway]{$\gamma^{-1}$ };
\end{tikzpicture} 
&
\begin{tikzpicture}[scale=.6]
\draw [ draw=gray, fill=gray, opacity=0.2]
       (0,0) -- (3,3) -- (3,-3) -- cycle;
\draw[-] (0,-3) -- (0,3)node[left]{\footnotesize Im};
\draw[-] (-3,0) -- (3,0)node[below]{\footnotesize Re};
\end{tikzpicture}
\end{tabular}

\begin{claim}\label{clm1} $\varphi_t$ satisfies condition \eqref{A} for $\alpha=1$ and $m=0$. That is to say for all $t\in (0,1)$, $\displaystyle{|\varphi^{(h)}_t(z)|\leq |\varphi'_t(0)|}$ for all $z\in \D$. 
\end{claim} 
\begin{proof}
Clearly, 
$\varphi_t(0)=0$ and
\[|\varphi'_t(z)| =\displaystyle{\frac{2t~ |\gamma(z)^{t-1}|~|\gamma'(z)|}{\left|\gamma(z)^t+1\right|^2}}.\] Since
$\displaystyle{\gamma'(z)=\frac{2}{(1-z)^2}}$, we see that $|\varphi'_t(0)|=t$.
It is known that image of $\varphi_t$ touches the boundary of the unit disk non tangentially at $1$ and $-1$. Now put $ w = \gamma(z) = re^{i\theta}$, we see that
\begin{align*}
|\varphi^{(h)}_t(z)|= &\frac{1-|z|^2}{1-\left|\frac{w^t-1}{w^t+1}\right|^2} ~\frac{2t~ |w^{t-1}|~|w'|}{\left|w^t+1\right|^2}\\
= &\frac{1-|z|^2}{|w^t+1|^2-|w^t-1|^2} ~2t~ |w^{t-1}|~|w'|.
\end{align*} 
On the other hand, we have
\begin{align*}
|w^t+1|^2-|w^t-1|^2=&(w^t+1)\overline{(w^t+1)}-(w^t-1)\overline{(w^t-1)} \\ = &(w^t+1)(\overline{w}^t+1)-(w^t-1)(\overline{w}^t-1)\\=&2(w^t+\overline{w}^t)\\=&2~r^t(e^{it\theta}+e^{-it\theta})\\=&4~r^t\cos t\theta. 
\end{align*}
And $w'=\gamma'(z)=\dfrac{2}{(1-z)^2}$
\[|\varphi^{(h)}_t(z)|=\displaystyle{\frac{1-|z|^2}{|1-z|^2}~\frac{t~r^{t-1}|e^{i(t-1)\theta}|}{r^t \cos t\theta}}.\] 
Using  $z=\dfrac{w-1}{w+1}$, we get 
\begin{align*}
|\varphi^{(h)}_t(z)|=&\displaystyle{\frac{1-\left|\frac{w-1}{w+1}\right|^2}{\left|1-\frac{w-1}{w+1}\right|^2}~\frac{t~r^{t-1}}{r^t \cos t\theta}}\\
&=\displaystyle{\frac{|w+1|^2-|w-1|^2}{4}~\frac{t~r^{t-1}}{r^t \cos t\theta}}\\
&=\displaystyle{\frac{4~r \cos\theta}{4}~\frac{t~r^{t-1}}{r^t \cos t\theta}}\\
&=\displaystyle{\frac{t \cos\theta}{ \cos t\theta}}.
\end{align*}
If $z\in (-1,1)$ then $\gamma(z)\in \mathbb{R_+}$. Therefore $\theta=0$ and so $|\varphi^{(h)}_t(z)|=t$. On the other hand if $z\in \D\setminus (-1,1)$ then $|\theta| \in (0,\pi/2)$. Hence $\cos t\theta > \cos\theta>0$ and so $|\varphi^{(h)}_t(z)|<t$. This completes the proof.
\end{proof}
\end{example} 

\begin{remark} From the proof of Claim \ref{clm1}, we see that $|\varphi^{(h)}_t(z)|\nrightarrow 0$ as $z$ approaches the boundary of the unit disk along the real-axis. Hence the composition operator with symbol $\varphi_t$ is a non-compact operator on $\bb$.
\end{remark}

The following proposition provides the sufficient condition that ensures the K\"onigs function belongs to Bloch-type spaces. This proposition plays an important role to prove main theorems. 
\begin{proposition}
\label{th13}
Assume $\cphi$ is bounded on $\bb_\alpha$ and $\varphi$ satisfies the condition 
\eqref{A} for some $\alpha>0$ and for some fixed non-negative integer $m$. Then $\sigma$ belongs to $\bb_\alpha$.
\end{proposition}

\begin{proof} Since $\cphi$ is bounded on $\bb_\alpha$,
 there exists positive number $M$ such that \begin{equation} \label{compth1}(1-|z|^2)^\alpha|\varphi'(z)| \leq M (1-|\varphi(z)|^2)^\alpha,\hspace{.5cm}~\text{for}~ z\in\D.\end{equation} 
For $m$ given by the assumption, choose non-negative integer $k$ such that $k>m$. For $z\in\D$, we have
\begin{align*}
(1-|z|^2)^\alpha|\varphi_k'(z)| 
 =&(1-|z|^2)^\alpha ~|\varphi'(\varphi_{k-1}(z))\varphi'(\varphi_{k-2}(z))...\varphi'(\varphi_{m-1}(z))\varphi'(\varphi_m(z))...\varphi'(z)| \\
  = & (1-|z|^2)^\alpha~|\varphi'(z)\varphi'(\varphi(z))~...\varphi'(\varphi_{m-1}(z))\varphi'(\varphi_m(z))...\varphi'(\varphi_{k-2}(z))~\varphi'(\varphi_{k-1}(z)|.
 \end{align*}
 By using \eqref{compth1},
 \begin{align*}
(1-|z|^2)^\alpha|\varphi_k'(z)| \leq & M(1-|\varphi(z)|^2)^\alpha~|\varphi'(\varphi(z))~...\varphi'(\varphi_{m-1}(z))\varphi'(\varphi_m(z))...\varphi'(\varphi_{k-2}(z))~\varphi'(\varphi_{k-1}(z)|.
\end{align*}
Again using \eqref{compth1} repeatedly, we get
\begin{align*}
(1-|z|^2)^\alpha|\varphi_k'(z)| \leq & M^m~(1-|\varphi_{m}(z)|^2)^\alpha |\varphi'(\varphi_{m}(z))...\varphi'(\varphi_{k-1}(z)| \end{align*}
Now using the condition \eqref{A} repeatedly, we get 
\begin{align*}
(1-|z|^2)^\alpha|\varphi_k'(z)| \leq & M^m |\varphi'(0)^{k-m}| ~(1-|\varphi_k(z)|^2)^\alpha.
  \end{align*}
  Thus, \[\lim_{k \rightarrow \8}(1-|z|^2)^\alpha~\left| \frac{\varphi'_k(z)}{\varphi'(0)^{k}}\right|    \leq \frac{M^m}{|\varphi'(0)^m|}\overline{\lim}_{k\rightarrow\8}~(1-|\varphi_k(z)|^2)^\alpha \leq \dfrac{M^m}{|\varphi'(0)^m|} . \]
This implies that 
 $  (1-|z|^2)^\alpha|\sigma'(z)|\leq \dfrac{M^m}{|\varphi'(0)^m|}$. Hence, $\sigma\in\bb_\alpha.$
\end{proof}
The following corollary provides a sufficient condition that ensures all the integer powers of the K\"onigs function belong to Bloch-type spaces $\bb_\alpha$ for $\alpha<1$.
\begin{theorem} \label{co15}
Suppose $\alpha <1$. If $\cphi$ is bounded on $\bb_\alpha$ and $\varphi$ satisfies the condition \eqref{A}, then $ \sigma^n\in\bb_\alpha$ for all positive integers n.
\end{theorem} 
\begin{proof}
From Proposition $\ref{th13}$, we see that $\sigma \in \bb_\alpha$. Suppose $\mathbb{H}^\8$ denotes the space of bounded holomorphic functions on the unit disk $\D$. Since $ \bb_\alpha \subset \mathbb{H}^\8$ for $\alpha <1$, so there exists a positive constant $C$   such that $\|\sigma\|_{\mathbb{H}^\8} \leq C.$  
 \begin{align*}
 (1-|z|^2)^\alpha |(\sigma^n(z))'| =&(1-|z|^2)^\alpha~|n~\sigma^{n-1}(z)~\sigma'(z)|\\  \leq &\|\sigma\|_{\bb_\alpha}~n~|\sigma^{n-1}(z)|  \\ \leq & n~\|\sigma\|_{\bb_\alpha}~  C^{n-1}.
\end{align*} 
Hence,  $\sigma^n \in \bb_{\alpha }$ for all positive integers $n$.
\end{proof}

The following theorem gives a sufficient condition that ensures all the integer powers of K\"onigs function belong to the Bloch space.
\begin{theorem}\label{thm2222}
Suppose $\varphi$ is a holomorphic self map of $\D$, $\varphi(0)=0$, $0<|\varphi'(0)|<1$. Also, assume that \begin{equation}
\label{eq12}
 \dfrac{1-|z|^2}{1-|\varphi(z)|^2}\dfrac{ \log \frac{2}{1-|z|}}{\log\frac{2}{1-|\varphi(z)|}} |\varphi'(z)|\leq |\varphi'(0)|,\hspace{.5cm}\text{for~all~} z\in \D.
\end{equation}  Then  $\cphi$ is bounded on $\bb$ and 
 $\sigma^n\in\bb$ for all positive integers $n$.
\end{theorem}

\begin{proof}
Boundedness of $\cphi$ on the Bloch space is consequence of Schwarz-Pick theorem. From the hypothesis of the theorem, we have
\begin{equation} \label{Blocheq}
(1-|z|^2)\log \frac{2}{1-|z|} |\varphi'(z)|\leq |\varphi'(0)| (1-|\varphi(z)|^2) \log\frac{2}{1-|\varphi(z)|},\hspace{.5cm}\text{for~all~} z\in \D.
\end{equation}
Suppose $k$ be a positive integer, then
\begin{align*}
(1-|z|^2)|\varphi_k'(z)| \log \frac{2}{1-|z|}
=&(1-|z|^2)|\varphi'(z)\varphi'(\varphi(z))......\varphi'(\varphi_{k-1}(z))|\log\frac{2}{1-|z|}\\
= &(1-|z|^2)\log\frac{2}{1-|z|}|\varphi'(z)\varphi'(\varphi(z))......\varphi'(\varphi_{k-1}(z))|.
\end{align*}
By using \eqref{Blocheq}, we see that
\begin{align*}
(1-|z|^2)|\varphi_k'(z)| \log \frac{2}{1-|z|} \leq &  |\varphi'(0)|(1-|\varphi(z)|^2)\log\frac{2}{1-|\varphi(z)|}|\varphi'(\varphi(z))......\varphi'(\varphi_{k-1}(z)|.
\end{align*}
And using \eqref{Blocheq} repeatedly, we get
\begin{align*}
(1-|z|^2)|\varphi_k'(z)| \log \frac{2}{1-|z|}
= & |\varphi'(0)|^k ( 1-|\varphi_k(z)|^2)\log\frac{2}{1-|\varphi_k(z)|}\\
\leq & 2 |\varphi'(0)|^k ( 1-|\varphi_k(z)|)\log\frac{2}{1-|\varphi_k(z)|}.\end{align*}
 \text{Since}~ $\log x\leq x ~\text{for}~ x>1$,
 \begin{align*}
 (1-|z|^2)|\varphi_k'(z)| \log \frac{2}{1-|z|}
\leq & 4 |\varphi'(0)|^k. 
\end{align*}
Hence,
\[\lim _{k\rightarrow\8}(1-|z|^2)\left|\dfrac{\varphi_k'(z)}{\varphi'(0)^k}\right| \log \dfrac{2}{1-|z|}=(1-|z|^2) |\sigma'(z)| \log \dfrac{2}{1-|z|} \leq 4, \hspace{.5cm}  z\in\D\] which shows that \begin{equation} \label{blocheq1}|\sigma'(z)|\leq \dfrac{4}{(1-|z|^2)  \log \dfrac{2}{1-|z|} }.\end{equation}
Recall that $\sigma(0)=0$. Now let us get an estimate for $\sigma$. 
\begin{align}\label{sigma1}
|\sigma(z)| = & \left| \int_o^1 \sigma'(tz) d(tz)\right|\nonumber\\
 \leq & \int_o^1  |\sigma'(tz)d(t|z|)\nonumber\\
\leq & \int_0^1 \dfrac{4}{\log \dfrac{2}{1-|tz| }} \dfrac{1}{1-|tz|^2} d(t|z|)\nonumber \\
 \leq&\;4~\left[\log\left(\log \dfrac{2}{1-t|z|}\right)\right]^1_0 \nonumber \\
 = & \;4 ~\left[\log \left(\log \dfrac{2}{1-|z|}\right)-\log(\log 2)\right].\end{align}
Now by using \eqref{blocheq1} and the estimate above for $\sigma$, we get  
\begin{align*}
(1-|z|^2)(\sigma^n(z))' =&(1-|z|^2) ~n~|\sigma^{n-1}(z)~\sigma'(z)| \\
\leq & \displaystyle{ 4^{n} n \left( \log \log \frac{2}{1-|z|}-\log\log 2 \right)^{n-1}\;\frac{1}{\log\frac{2}{1-|z|}}}.
\end{align*}
 Taking limit $|z|\rightarrow1$, it is easy to see that the right hand side of the last expression goes to zero. Hence  $\sigma^n\in  \bb$ for all positive integers $n$.

\end{proof}
 Let us recall the Lipschitz-type norm which is equivalent to the usual norm defined for function $f\in \bb_\alpha$, $\alpha >1$: 
\[\|f\|_{\bb_\alpha} \equiv \sup_{z\in\D} (1-|z|^2)^{\alpha-1}|f(z)|.\]
Next, we present results for the Bloch-type spaces, $\bb_\alpha$ for $\alpha>1$. Let us start with the following definition.

\begin{definition}
Suppose $f\in \bb_\alpha$ for some $\alpha>0$, then we define the {\it Bloch number} of $f$ by 
  $b_f ~=\displaystyle{\inf_{\alpha}~ \{\alpha: f\in \bb_\alpha\}}$.
\end{definition}

\begin{proposition}
\label{th113}
Suppose $\beta>0$. Then $f^n\in\bb_{\beta+1}$ for all positive integers $n$ if and only if $b_f$  is at most 1.
\end{proposition}
\begin{proof}
Suppose $f^n \in \bb_{\beta +1}$ for all positive integers $n$.
Need to show $b_f\leq 1$. On the contrary assume $b_f>1$. Then there exists a positive integer $n_0$ such that  $1< 1+\frac{\beta}{n_0} < b_f$.
 Now in the view of the Lipschitz-type norm, we see that for any fixed positive integer $M$ there exists $z\in\D$ such that
 \begin{align*}
M\leq (1-|z|^2)^{\beta/n_0}|f (z)| \leq & \{(1-|z|^2)^{\beta/n_0}|f (z)|\}^{n_0}\\
=& (1-|z|^2)^{\beta}|f (z)|^{n_0}, 
\end{align*}
 which shows that
  \[M\leq\sup_{z\in\D} (1-|z|^2)^{\beta}|f (z)|^{n_0}=\|f^{n_0}\|_{\bb_{\beta+1}}.\]
 
Since $M$ is an arbitrary positive integer, $f^{n_0} \notin \bb_{\beta+1}$. Which is a contradiction. 

Conversely, suppose $b_f\leq 1$. Since $\bb_\alpha\subset \bb$ for all $\alpha \leq 1$, then clearly $f\in\bb$. For any fixed $\beta >0$ and for any fixed positive integer $n$,
\begin{align*}
(1-|z|^2)^{\beta+1}|(f^n)'(z)|
= & (1-|z|^2)^{\beta+1}|n f^{n-1}(z) f'(z)|\\
= & n (1-|z|^2) |f'(z)| (1-|z|^2)^{\beta} | f^{n-1}(z)|\\
 \leq   & n \|f\|_{\bb} (1-|z|^2)^{\beta} \left(\|f\|_{\bb} \log\frac{1}{1-|z|}\right)^{n-1}\\
= & n (\|f\|_{\bb})^n (1-|z|^2)^{\beta} \left( \log\frac{1}{1-|z|}\right)^{n-1}.
\end{align*}
The last expression goes to zero as $|z| \rightarrow 1$.
 This shows that $f^n\in\bb_{\beta+1}$ for all positive integers $n$.
\end{proof}
\begin{theorem} \label{thm3.3}
Suppose $\varphi$ is a holomorphic self map of $\D$, $\varphi(0)=0$, $0<|\varphi'(0)|<1$, and also assume $\alpha>1$. If $|\varphi^{(h)}(z)|\leq |\varphi'(0)|$ for all  $z\in \D $ then $\cphi$ is bounded on $\bb_\alpha$ and $\sigma^n\in\bb_\alpha$ for all positive integers $n$
\end{theorem}
\begin{proof}
Since $|\varphi^{(h)}(z)|\leq |\varphi'(0)|$ for all  $z\in \D $,  from Proposition $\ref{th13}$ $\sigma \in \bb$. So $b_f\leq 1$. Therefore the result follows from Proposition $\ref{th113}$.
\end{proof}


\section{Weighted Composition operator}\label{wtcompopt}

Let us recall that if $ u $ is a holomorphic function of the unit disk, and $\varphi$ is a holomorphic self map of the unit disk then the Schr\"oder equation for weighted composition operator is given by
\begin{equation}
\label{eq 21}
 u(z)f(\varphi(z)) =\lambda f(z),
 \end{equation}
 where $f\in\h$ and $\lambda$ is a complex constant.
 
Let us also recall that if $u(0)\neq 0, \varphi(0)=0, 0<|\varphi'(0)|<1$ then the solutions of \eqref{eq 21} are given by  Theorem \ref{th22}. The principal eigenfunction corresponding to the eigenvalue $u(0)$ is denoted by $v$ and all other eigenfunctions are of the form $v\sigma^n$  where $\sigma$ is the K\"onigs function of $\varphi$ and $n$ is a positive integer. Hosokawa and Nguyen \cite{Ho10} studied equation \eqref{eq 21} in the Bloch space and obtained the following result. 
\begin{theorem} 
\label{th23}
Assume $\varphi$ is a holomorphic self map of $\D$ with $\varphi(0)=0$ and $0<|\varphi'(0)|<1$,  and $u$ is holomorphic map of $\D$ such that $u(0)\neq 0$. Let us also assume that $u\cphi$ is bounded on $\bb$. For $0 < r < 1$, set
\[M_r(\varphi):=\sup_{|z|=r}|\varphi(z)|\hspace{1cm}a_r:= \sup_{|z|=r} \big(|u '(z)\varphi(z)| + |u(z)\varphi'(z)|\big).\]Suppose that
\begin{enumerate}
\item[(i)] $\lim_{r\rightarrow 1} \log(1 - r) \log M_r(\varphi) = \8.$
\item[(ii)] $\log |a_r| < \epsilon \log(1 -r) \log M_r(\varphi),$ \\where $\epsilon>0$ is a constant satisfying $\epsilon  \log \Vert \varphi\Vert_\8> -1.$
\end{enumerate}
Then, $v\sigma^n\in\bb$ for all non-negative integer $n$. 
\end{theorem}

Here we investigate the properties of weight $u$ and symbol $\varphi$ of weighted composition operators $u\cphi$ that ensure $v\sigma^n$ belongs to Bloch-type spaces $\bb_\alpha$ for some $\alpha>0$ and for all non negative integer $n$. Let us begin with following remark.
\begin{remark} \label{rem1}
Suppose $f$ is a holomorphic function defined on $\D.$ If $\|f'\|_\8<M$ for some $M>0$ then
\begin{align*}
|f(z)-f(0)|=& \left|\int_0^1z f'(tz)dt\right| \\
\leq & \int_0^1|z~ f'(tz)|dt \\
\leq  &~ M\int_0^1|z|dt
\end{align*}
If $f$ also satisfies $f(0)=0$, then $\|f\|_\8\leq M$. 

\end{remark}

\begin{proposition} Assume $\varphi$ is a univalent holomorphic self map of the unit disk with $\varphi(0)=0$ and $0<|\varphi'(0)|<1$, and $\sigma$ is K\"onigs function of $\varphi$. Then, $\sigma$ is bounded if and only if there is a positive integer $k$ such that $\|\varphi_k\|_\8<1$. 
\end{proposition}

\begin{proof}
Suppose $\sigma$ is bounded. Since $\varphi$ is univalent, $\sigma$ is also univalent (see \cite{Sh93}, page 91). Since $\sigma$ is bounded univalent map, there is a positive integer $k$ such that $\|\varphi_k\|_\8<1$ (see \cite{Sh93}).\\
Conversely suppose there is a positive integer $k$ such that $\|\varphi_k\|_\8<1$. Since we have $\sigma(\varphi(z))=\varphi'(0) \sigma(z)$, 
\begin{align*}
\sigma(\varphi_k(z))= & \sigma(\varphi(\varphi_{k-1}(z))\\=& \varphi'(0) \sigma(\varphi_{k-1}(z))\\=&\varphi'(0)^k \sigma(z).
\end{align*} Clearly left hand side is bounded and therefore $\sigma$ is also bounded, which completes the proof.
\end{proof}

\begin{theorem}
Assume $\varphi$ is a univalent holomorphic self map of the unit disk with $\varphi(0)=0$ and $0<|\varphi'(0)|<1$ satisfying  $|\varphi^{(h_\alpha)}(z)|\leq |\varphi'(0)|$ for all $z\in \D $ and for some fixed $\alpha<1$. If $u$ is holomorphic map of $\D$ such that $u(0)\neq 0$ and $\|u'\|_\8<\8$ then $u\cphi$ is bounded on $\bb_\alpha$ and $v\sigma^n\in\bb_\alpha$ for all non-negative integers $n$. 
\end{theorem}
\begin{proof} Since $\|u\|_\infty<\|u'\|_\infty+|u(0)| <\infty$ and $|\varphi^{(h_\alpha)}(z)|\leq |\varphi'(0)|$, $u\cphi$ is bounded on $\bb_\alpha$  for some $\alpha<1$.

Since $|\varphi^{(h_\alpha)}(z)|\leq |\varphi'(0)|$ for some $\alpha<1$, using Proposition \ref{th13}, we see that $\sigma\in\bb_\alpha$, $\alpha<1$ and hence bounded. Since $\varphi$ is univalent, $\sigma$ is univalent. Consequently, there exists a non-negative integer $k$ such that $\|\varphi_k\|_\8<1$.
Composing $\varphi_{k-1} $ on both sides of the Schr\"oder equation \eqref{eq 21} from right,
\begin{equation} \label{alpha 1}
 u(\varphi_{k-1}(z))f(\varphi_{k}(z)) =\lambda f(\varphi_{k-1}(z)). 
\end{equation}
The left hand side in the equation above is bounded and so is $ f\circ \phi_{k-1}$. Now differentiating both side of \eqref{alpha 1}, we get that
\[ u'(\varphi_{k-1}(z))~\varphi'_{k-1}(z)~ f(\varphi_{k}(z)) + u(\varphi_{k-1}(z)) ~f'(\varphi_k(z))~\varphi'_k(z) 
=\lambda f'(\varphi_{k-1}(z))~\varphi'_{k-1}(z).\]
Multiplying both sides by $(1-|z|^2)^\alpha$ and using boundedness of 

 $\|u'\|_\8$, $\|u\|_\8$, $ f\circ \varphi_k$ and $f'\circ \varphi_k$, we see that there exists a constant $M$ such that
\begin{equation}\label{alpha 2}
(1-|z|^2)^\alpha|\lambda f'(\varphi_{k-1}(z))\varphi'_{k-1}(z)|\leq M(1-|z|^2)^\alpha\big(|\varphi'_{k-1}(z)|  +| \varphi'_k(z)|\big).\end{equation}
Right hand side of the above equation is uniformly bounded and therefore the left hand side is bounded.  
Again, let us compose $\varphi_{k-2}$ on \eqref{eq 21}, to get
\[ u(\varphi_{k-2}(z))f(\varphi_{k-1}(z)) =\lambda f(\varphi_{k-2}(z)).\]
Let us differentiate above expression and then multiply by $(1-|z|^2)^\alpha$ on both sides. Then, the use of \eqref{alpha 1} and \eqref{alpha 2} shows that $(1-|z|^2)^\alpha| f'(\varphi_{k-2}(z))\varphi'_{k-2}(z)|$ is bounded. 

Continuing this process, we see that that $ \sup_{z\in\D}(1-|z|^2)^\alpha |f'(z)|$ is bounded and hence  $f\in\bb_\alpha$. From Theorem \ref{th22}, we know that any holomorphic $f$ satisfying \eqref{eq 21} is of the form $v\sigma^n$ for some positive integer $n$, so $v\sigma^n\in \bb_\alpha$ for all non negative integers $n$. This completes the proof.

\end{proof}

  The following two theorems give us the sufficient conditions that ensure $v\sigma^n$ belong to Bloch-type spaces $\bb_\alpha$ for some $\alpha>1$ and for all non-negative integers $n$. 
\begin{theorem}
Let $\varphi$ be a holomorphic self map of the unit disk with $\varphi(0)=0$ and $0<|\varphi'(0)|<1$,  and $u$ is holomorphic map of $\D$ such that $u(0)\neq 0$. Let $\beta$ be a fixed positive number and assume   \[|u(z)|\frac{(1-|z|^2)^{\beta}}{(1-|\varphi(z)|^2)^{\beta}}\leq |u(0)|,\hspace{.5cm}\text{for~all}~z\in\D.\]  Then the following statements are true. 

\renewcommand\theenumi{\roman{enumi}}
\begin{enumerate}
\item If  $\displaystyle{ |\varphi^{(h_\alpha)}(z)|\leq |\varphi'(0)|}$ for all $z\in\D$ and for some $\alpha<1$, then $v\sigma^n\in \bb_{\beta+1}$ for all non-negative integers $n$.

\item If $ |\varphi^{(h)}(z)|\leq |\varphi'(0)|$ for all $z\in\D$, then $v\sigma^n\in \bb_{p+1}$, for some $p>\beta$ and for all non-negative integers $n$.
\end{enumerate}

\end{theorem}
\begin{proof}
\renewcommand\theenumi{\roman{enumi}}
\begin{enumerate}
\item 
From definition of $v_k$ (see Theorem \ref{th22}), we have

\begin{align*}
(1-|z|^2)^{\beta} |v_k(z)|=& \displaystyle{(1-|z|^2)^{\beta}\frac{|u(z)u(\varphi(z))......u(\varphi_{k-1}(z))|}{|u(0)|^k}} \\ \leq & (1-|\varphi(z)|^2)^{\beta}\frac{|u(\varphi(z))......u(\varphi_{k-1}(z))|}{|u(0)|^{k-1}}\\ &...\\\leq & 1.
\end{align*}
Hence $ (1-|z|^2)^{\beta} |v(z)|=\lim_{k\rightarrow \8} (1-|z|^2)^{\beta} |v_k(z)|\leq 1.$ Since $z$ is arbitrary,  
\[\sup_{z\in\D}(1-|z|^2)^{\beta} |v(z)|<\8.\] On the other hand the assumption   $\displaystyle{ |\varphi^{(h_\alpha)}(z)|\leq |\varphi'(0)|}$ and Proposition \ref{th13} implies that $ \sigma^n\in\bb_\alpha\subset\mathbb{H}^\8$ for all non-negative integer $n$. Hence, \[\sup_{z\in\D}(1-|z|^2)^{\beta} |v(z) \sigma^n(z)|<\8\] for all non-negative integers $n$. Considering the equivalent norm (see \eqref{PRD1}), we see that $v\sigma^n\in\bb_{\beta+1}$ for all non-negative integers $n$.
 
\item From the proof of $(i)$, we see that \begin{equation} \label{WCE10}
\sup_{z\in\D}(1-|z|^2)^{\beta} |v(z)|<\8.
\end{equation}  On the other hand, since $ |\varphi^{(h)}(z)|\leq |\varphi'(0)|$, Proposition \ref{th13} implies that $ \sigma\in \bb$ so there exists $M>0$ such that \begin{equation}
\label{WCE20}
|\sigma(z)|\leq M\log\frac{2}{1-|z|^2}.\end{equation} Now using equations \eqref{WCE10} and \eqref{WCE20}, we have \begin{align*}
(1-|z|^2)^{p} |v(z) \sigma^n(z)|= & \{(1-|z|^2)^\beta |v(z|)\}\{(1-|z|^2)^{p-\beta} |\sigma^n(z)|\} \\ \leq& C M (1-|z|^2)^{p-\beta}\left(\log\frac{2}{1-|z|^2}\right)^n 
\end{align*}  for some constant $C>0$. Now taking limit $|z|\rightarrow 1$, we get that the last expression goes to zero. Hence, $v\sigma^n\in \bb_{p+1}$ for all non-negative integers $n$.
\end{enumerate}

\end{proof}

\begin{theorem}
Let $\varphi$ be a holomorphic self map of the unit disk with $\varphi(0)=0$ and $0<|\varphi'(0)|<1$, and $u$ is holomorphic map of $\D$ such that $u(0)\neq 0$. Suppose that $\beta$ is a positive integer and 
 \renewcommand\theenumi{\roman{enumi}}
\begin{enumerate}
\item
 $\displaystyle{|u(z)|\frac{(1-|z|^2)^{\beta}}{(1-|\varphi(z)|^2)^{\beta}} \frac{\log  \frac{2}{(1-|z|)^\beta}}{\log \frac{2}{(1-|\varphi(z)|)^\beta}}\leq |u(0)|} ,\hspace{.5cm} \text{for~ all~} z\in\D$
\item 
$\displaystyle{|\varphi^{(h)}(z)| \frac{\log  \frac{2}{1-|z|}}{\log \frac{2}{1-|\varphi(z)|}}\leq |\varphi'(0)|},\hspace{.5cm} \text{for~ all~} z\in\D.$
\end{enumerate} 
Then $ v\sigma^n\in\bb_{\beta+1}$ for all non-negative integers $n$.
\end{theorem}

\begin{proof} From definition of $v_k$ on \ref{th22}) and the condition $(i)$, we have
 \begin{align*}
(1-|z|^2)^{\beta} \log  \frac{2}{(1-|z|)^\beta} |v_k(z)|= &\displaystyle{(1-|z|^2)^{\beta}\log  \frac{2}{(1-|z|)^\beta}\frac{|u(z)u(\varphi(z))......u(\varphi_{k-1}(z))|}{|u(0)|^k}}\\ \leq &\displaystyle{(1-|\varphi(z)|^2)^{\beta}\log  \frac{2}{(1-|\varphi(z)|)^\beta}\frac{|u(\varphi(z))......u(\varphi_{k-1}(z))|}{|u(0)|^{k-1}}}\\ \leq& (1-|\varphi_k(z)|^2)^{\beta}\log\frac{2}{(1-|\varphi_k(z)|)^\beta}\\\leq& 2^\beta(1-|\varphi_k(z)|)^{\beta} \log\frac{2}{(1-|\varphi_k(z)|)^\beta}.
\end{align*}

\text{Since} ~~$ \log x\leq x,~~ \text{for}~~ x>1$

\[(1-|z|^2)^{\beta} \log  \frac{2}{(1-|z|)^\beta} |v_k(z)|\leq  2^{\beta+1}.\]

So taking limit $k$ approaches to $\8$, we see that \begin{equation}
\label{WCE30}
 (1-|z|^2)^{\beta}|v(z)|\leq \displaystyle{\frac{2^{\beta+1}}{\log  \frac{2}{1-|z|}}}.\end{equation}
On the other hand, since $\varphi$ satisfies condition $(ii)$, equation \eqref{sigma1} of Theorem \ref{thm2222} says that there exists $K>0$ such that \begin{equation}
\label{WCE40} |\sigma(z)| \leq  K \displaystyle{\log\log  \frac{2}{1-|z|}}.\end{equation}
Now using \eqref{WCE30} and \eqref{WCE40}, we get \[(1-|z|^2)^{\beta}|v(z)\sigma^n(z)|\leq \displaystyle{\frac{2^{\beta+1}K^n}{\log  \frac{2}{1-|z|}}  \left(\log\log  \frac{2}{1-|z|}\right)^n}.\] Clearly right hand side of the above equation goes to 0 as $|z|\rightarrow 1 $. Using the norm defined on \eqref{PRD1}, $v\sigma^n\in\bb_{\beta+1}$ for all non negative integers $n$.

\end{proof}

 This paper is based on a research which forms a part of the author's Ph.D. dissertation from University of Toledo. The author wishes to express his deep gratitude to his dissertation adviser Professor \v{Z}eljko \v{C}u\v{c}kovi\'c.

\begin{center}

Bhupendra Paudyal\\
Central State University\\
Math and Computer Sc. Dept.\\ Wilberforce, Ohio, USA\\

{\it Email: bpaudyal@centralstate.edu}

\end{center}

\end{document}